\documentclass[a4paper,12pt,reqno]{amsart}

\usepackage{amsmath, amsfonts, amssymb, amsthm, mathrsfs, mathtools, mathalfa}
\usepackage[export]{adjustbox}
\usepackage{comment}
\usepackage{yfonts}
\usepackage{enumerate}
\usepackage{enumitem}
\usepackage{stmaryrd}
\usepackage{fancyhdr}
\usepackage{bm}
\usepackage[utf8]{inputenc}
\usepackage[pdfencoding=auto, hyperfootnotes=false]{hyperref}
\usepackage{tikz,tikz-cd}
\usepackage[hang,flushmargin]{footmisc}
\usepackage{graphicx}
\usepackage{xr}
\usepackage[dvistyle, colorinlistoftodos]{todonotes}
\usepackage{extarrows}

\usepackage{thmtools}
\usepackage{thm-restate}
\usepackage{hyperref}

\makeatletter
\def\@tocline#1#2#3#4#5#6#7{\relax
  \ifnum #1>\c@tocdepth 
  \else
    \par \addpenalty\@secpenalty\addvspace{#2}%
    \begingroup \hyphenpenalty\@M
    \@ifempty{#4}{%
      \@tempdima\csname r@tocindent\number#1\endcsname\relax
    }{%
      \@tempdima#4\relax
    }%
    \parindent\z@ \leftskip#3\relax \advance\leftskip\@tempdima\relax
    \rightskip\@pnumwidth plus4em \parfillskip-\@pnumwidth
    #5\leavevmode\hskip-\@tempdima
      \ifcase #1
       \or\or \hskip 1em \or \hskip 2em \else \hskip 3em \fi%
      #6\nobreak\relax
    \dotfill\hbox to\@pnumwidth{\@tocpagenum{#7}}\par
    \nobreak
    \endgroup
  \fi}
\makeatother

\newtheorem{theorem}{Theorem}[section]
\newtheorem{lemma}[theorem]{Lemma}

\newtheorem{proposition}[theorem]{Proposition}

\theoremstyle{definition}
\newtheorem{defn}[theorem]{Definition}
\newtheorem{remark}[theorem]{Remark}
\newtheorem{example}[theorem]{Example}

\newcommand{\mc}{\mathcal}

\newcommand{\mf}{\mathbf}
\newcommand{\mb}{\mathbb}

\newcommand{\wh}{\widehat}

\newcommand{\id}{\mathrm{id}}

\DeclareMathOperator{\ab}{Z}

\DeclareMathOperator{\tran}{\Theta}

\DeclareMathOperator{\diam}{diam}

\DeclareMathOperator{\q}{c}
\DeclareMathOperator{\ns}{X}
\DeclareMathOperator{\nss}{Y}
\DeclareMathOperator{\co}{\circ\hspace{-0.02 cm}}
\DeclareMathOperator{\cu}{C}
\DeclareMathOperator{\cor}{Cor}

\newcommand*{\db}[1]{\llbracket #1\rrbracket}

\begin{document}

\title[On nilspace systems and their morphisms]{On nilspace systems and their morphisms}

\author{Pablo Candela}
\address{Universidad Aut\'onoma de Madrid and ICMAT\\ Ciudad Universitaria de Cantoblanco\\ Madrid 28049\\ Spain}
\email{pablo.candela@uam.es}

\author{Diego Gonz\'alez-S\'anchez}
\address{Universidad Aut\'onoma de Madrid and ICMAT\\ Ciudad Universitaria de Cantoblanco\\ Madrid 28049\\ Spain}
\email{diego.gonzalezs@predoc.uam.es}

\author{Bal\'azs Szegedy}
\address{MTA Alfr\'ed R\'enyi Institute of Mathematics\\ 
Re\'altanoda utca 13-15\\
Budapest, Hungary, H-1053}
\email{szegedyb@gmail.com}

\begin{abstract}
A nilspace system is a generalization of a nilsystem, consisting of a compact nilspace $\ns$ equipped with a group of nilspace translations acting on $\ns$. Nilspace systems appear in different guises in several recent works, and this motivates the study of these systems per se as well as their relation to more classical types of systems. In this paper we study morphisms of nilspace systems, i.e., nilspace morphisms with the additional property of being consistent with the actions of the given translations. A nilspace morphism does not necessarily have this property, but one of our main results shows that it factors through some other morphism which does have the property. As an application we obtain a strengthening of the inverse limit theorem for compact nilspaces, valid for nilspace systems. This is used in work of the first and third named authors to generalize the celebrated structure theorem of Host and Kra on characteristic factors. \smallskip 
\end{abstract}
\date{}
\vspace*{-1.5cm}
\maketitle
\vspace{-0.5cm}
\section{Introduction}

\noindent Nilsystems are important examples of measure-preserving systems and their study has a long history in ergodic theory, beginning with works including \cite{AGH,Furstentorus,Malcev,Parry}. This study gained strong motivation especially through the essential role of nilsystems in the structural theory of measure-preserving systems and the analysis of multiple ergodic averages, topics that have kept growing with vibrant progress to the present day. We refer to the book \cite{HKbook} for a thorough treatment of these rich topics and for a broad bibliography.

In recent works stemming from the connections between ergodic theory and arithmetic combinatorics, objects known as \emph{compact nilspaces} are found to be useful generalizations of nilmanifolds. Similarly to how nilsystems are constructed using nilmanifolds, one can define  generalizations of nilsystems using compact nilspaces, thus obtaining measure-preserving systems that we call \emph{nilspace systems}. These systems emerge as natural objects to consider when trying to extend the structural theory of measure-preserving systems, or that of topological dynamical systems, beyond works such as that of Host and Kra \cite{HK}, Ziegler \cite{Ziegler}, or Host Kra and Maass \cite{HKM}, and especially when seeking extensions valid for nilpotent group actions; see \cite{CScouplings,GGY}.

The theory of nilspaces is growing into a subject of intrinsic interest. In addition to the original preprint \cite{CamSzeg}, there are now several references that detail the basics of this theory; see  \cite{Cand:Notes1,Cand:Notes2,GMV1,GMV2}. To state the definition of a nilspace system below, we use the notions of a compact nilspace $\ns$ and of the translation group $\tran(\ns)$, which can be recalled from \cite[Definition 1.0.2]{Cand:Notes2} and \cite[Definition 3.2.27]{Cand:Notes1} respectively. Recall also that, on a \emph{compact} nilspace, translations are supposed to be homeomorphisms; see \cite[\S 2.9]{Cand:Notes2}.

\begin{defn}
A \emph{nilspace system} is a triple $(\ns, H, \phi)$ where $\ns$ is a compact nilspace, $H$ is a topological group, and $\phi:H \to \tran(\ns)$ is a continuous group homomorphism. We say that $(\ns, H, \phi)$ is a \emph{$k$-step nilspace system} if $\ns$ is a $k$-step nilspace.
\end{defn}
\noindent Nilspace systems are indeed generalizations of nilsystems: given a nilmanifold $G/\Gamma$ and a map $T:G/\Gamma\to G/\Gamma$, $x\Gamma\mapsto hx\Gamma$ for some $h\in G$, it is seen from the definitions that $T$ is a translation on $\ns$, where $\ns$ is the nilspace obtained by endowing $G/\Gamma$ with the cube structure induced by the Host-Kra cubes $\cu^n(G_\bullet)$ relative to any given filtration $G_\bullet$ on $G$ (see \cite[Definition 2.2.3 and Proposition 2.3.1]{Cand:Notes1}). A nilspace system can be viewed as a measure-preserving system, by equipping the compact nilspace with its Haar probability measure, which is invariant under any translation; see \cite{CamSzeg} and \cite[Proposition 2.2.5 and Corollary 2.2.7]{Cand:Notes2}. The term \emph{nilsystem} can also be used more generally, when instead of a single map $T$ we have an action of a group $H$ (usually supposed to be countable and discrete) on $G/\Gamma$ via a homomorphism $\varphi:H\to G$, action defined by $(h,x\Gamma)\mapsto \varphi(h)x \Gamma$.

It is natural to seek expressions for nilspace systems in terms of nilsystems, so as to reduce questions involving the former systems to questions involving the better-known latter systems. One of the central results in nilspace theory is the inverse limit theorem \cite[Theorem 4]{CamSzeg}; in particular this result characterizes a general class of compact nilspaces (those with connected structure groups) as inverse limits of nilmanifolds. This motivates the problem of expressing nilspace systems as inverse limits of nilsystems. Similar problems are treated by Gutman, Manners and Varj\'u in \cite{GMV3} from the different viewpoint of applications in topological dynamics, concerning the \emph{regionally proximal relations}. Further motivation comes from the use  of nilspace systems in \cite{CScouplings} to extend the structure theorem of Host and Kra \cite[Theorem 10.1]{HK} to nilpotent group actions.

To obtain such expressions of nilspace systems in terms of nilsystems, it is suitable first to focus on certain more fundamental questions concerning nilspace systems in themselves, and especially on how translations on a compact nilspace $\ns$ can interact with continuous\footnote{In this paper every morphism between compact nilspaces is automatically supposed to be a continuous map,  and from now on we usually do not specify this continuity explicitly.} morphisms from $\ns$ to another compact nilspace. As we show in this paper, once these questions are solved, the sought expressions for nilspace systems can be obtained swiftly. 

One fundamental question of the above kind asks whether a given nilspace morphism satisfies the following property relative to a given translation.
\begin{defn}\label{def:transcons}
Let $\ns,\nss$ be nilspaces, let $\psi:\ns\to\nss$ be a morphism, and let $\alpha$ be a translation in $\tran(\ns)$. We say that $\psi$ is \emph{$\alpha$-consistent} if for every $x,y\in \ns$ we have $\big( \psi(x)=\psi(y)\big) \Rightarrow \big(\psi\co\alpha (x)=\psi\co\alpha(y)\big)$. Given a set of translations $H\subset\tran(\ns)$, we say that $\psi$ is \emph{$H$-consistent} if $\psi$ is $\alpha$-consistent for each $\alpha\in H$.
\end{defn}
\noindent The question of whether a morphism is $\alpha$-consistent is particularly relevant for the special class of morphisms termed \emph{fiber-surjective morphisms}. Introduced in \cite{CamSzeg} (see also  \cite[Definition 3.3.7]{Cand:Notes2}), these morphisms play an important role in nilspace theory. The term \emph{fibration} was introduced in \cite[Definition 7.1]{GMV1} for a notion which is equivalent to that of a fiber-surjective morphism as far as nilspaces are concerned, and which gives a useful alternative definition; we shall use the two terms interchangeably. We recall these notions in Definition \ref{def:fsmorph} below. This definition uses the characteristic factors $\ns_n=\mc{F}_n(\ns)$ of a nilspace $\ns$, and the associated canonical projections $\pi_n:\ns\to\ns_n$ (see \cite[Definition 3.2.3]{Cand:Notes1}). When we need to emphasize on which nilspace $\ns$ the map $\pi_n$ is being considered, we denote this map by $\pi_{n,\ns}$. We also use the notation $\db{n}$ for the discrete $n$-cube $\{0,1\}^n$, and $\cu^n(\ns)$ for the set of $n$-cubes on $\ns$. Finally, let us recall the notion of an $n$-corner on $\ns$, that is, a map $\q':\db{n}\setminus\{1^n\}\to \ns$ (where $1^n=(1,\ldots,1)$) such that the restriction of $\q'$ to any $(n-1)$-face of $\db{n}$ not  containing $1^n$ is an $(n-1)$-cube (see \cite[Definition 1.2.1]{Cand:Notes1}). We denote the set of $n$-corners on $\ns$ by $\cor^n(\ns)$.

\begin{defn}\label{def:fsmorph}
Let $\ns,\nss$ be nilspaces. A morphism $\psi:\ns\to\nss$ is said to be \emph{fiber-surjective}, or a \emph{fibration}, if for every $n\geq 0$ it maps $\pi_n$-fibers to $\pi_n$-fibers, that is, for every fiber $\pi_{n,\ns}^{-1}(x)$, $x\in \ns_n$, we have $\psi(\pi_{n,\ns}^{-1}(x))=\pi_{n,\nss}^{-1}(y)$ for some $y\in\nss_n$. Equivalently $\psi$ is a fibration if for every $n$-corner $\q'$ on $\ns$, and every completion $\q''\in \cu^n(\nss)$ of the $n$-corner $\psi\co\q'$, there is $\q\in\cu^n(\ns)$ completing $\q'$ such that $\psi\co\q=\q''$.
\end{defn}
\begin{remark}\label{rem:fibration} Note that a fibration is always a surjective map, since for every $k$-step nilspace $\nss$ the fiber of $\pi_0$ is the whole of $\nss$ (note that $\nss_0$ is a singleton).
\end{remark}
\noindent The equivalence stated in Definition \ref{def:fsmorph} follows from the properties of the equivalence relations $\sim_n$ associated with the quotient maps $\pi_n$ (see \cite[Definition 3.2.3]{Cand:Notes1}).

The following simple result illustrates the relevance of $\alpha$-consistency for fibrations.
\begin{restatable}{lemrestate}{basic}\label{lem:basic}
\textit{Let $\ns,\nss$ be compact nilspaces, let $\psi:\ns\to\nss$ be a fibration, and let $\alpha \in \tran(\ns)$. If $\psi$ is $\alpha$-consistent then we can define $\beta\in\tran(\nss)$ by setting $\beta(y)=\psi(\alpha(x))$ for any $x\in \psi^{-1}(y)$. Moreover, if for every translation $\alpha\in \tran(\ns)$ such that $\psi$ is $\alpha$-consistent, we denote by $\wh{\psi}(\alpha)$ the corresponding translation $\beta\in\tran(\nss)$ thus defined, then, whenever $\psi$ is $\{\alpha_1,\alpha_2\}$-consistent, we have that $\psi$ is $\alpha_1\alpha_2$-consistent and $\wh{\psi}(\alpha_1\alpha_2)=\wh{\psi}(\alpha_1)\,\wh{\psi}(\alpha_2)$.}
\end{restatable}
\noindent In particular, if $S$ is a subset of $\tran(\ns)$ such that $\psi$ is $S$-consistent, and $H=\langle S\rangle$ is the subgroup of $\tran(\ns)$ generated by $S$, then $\psi$ is also $H$-consistent, the map $\wh{\psi}$ is a homomorphism $H\to \tran(\nss)$, and $\wh{\psi}$ is an $(H,\wh{\psi}(H))$-equivariant map, i.e., for all $\alpha\in H$, $x\in \ns$ we have $\psi(\alpha(x))=\wh{\psi}(\alpha)(\psi(x)
)$. We leave the proof of Lemma \ref{lem:basic} to Section \ref{sec:transconsfact}.

To what extent can $\alpha$-consistency be guaranteed for a given fibration and a given translation $\alpha$? If $\ab,\ab'$ are compact abelian groups equipped with their standard nilspace structure (see \cite[Section 2.1]{Cand:Notes1}), and $\psi:\ab\to\ab'$ is a fibration between these nilspaces, then $\psi$ is a surjective affine homomorphism (by \cite[Lemma 3.3.8]{Cand:Notes1} say), and must then clearly be $\alpha$-consistent for every $\alpha\in\tran(\ab)$ (since $\alpha$ must be of the form $\alpha(z)= z+t$ for some fixed $t\in \ab$). However, this automatic translation-consistency does not hold for fibrations between more general nilspaces; we illustrate this with Example \ref{ex:conscex} in the next  section.

While translation consistency can fail, we  prove a result that can be viewed as the next best thing, namely Theorem \ref{thm:trancons} below. Indeed, this result tells us that by performing a relatively simple refinement of the fibration's target nilspace, one can always obtain a factorization of the fibration in which the first applied map has the desired consistency. The simplicity of the refinement consists in that, if the original target nilspace is of finite rank, then the new refined nilspace is also of finite rank. Recall that a compact nilspace has finite rank if each of its structure groups is a Lie group (see \cite[Definition 2.5.1]{Cand:Notes2}). These nilspaces form a more  manageable class among general compact nilspaces, playing a role in the theory similar to the role of compact abelian Lie groups in relation to general compact abelian groups (see for instance the inverse limit theorem \cite[Theorem 2.7.3]{Cand:Notes2}).

\begin{restatable}{thmrestate}{trancons}\label{thm:trancons}
\textit{Let $\ns, \nss'$ be compact $k$-step nilspaces, with $\nss'$ of finite rank, let $\psi':\ns\to \nss'$ be a fibration, and let $H\subset \tran(\ns)$ be finite. Then there is a compact finite-rank nilspace $\nss$, an $H$-consistent fibration $\psi:\ns\to \nss$, and a fibration $p: \nss\to\nss'$ such that $\psi'=p\co\psi$.}
\end{restatable}
\noindent This theorem is our main result in Section \ref{sec:transconsfact}. It relies on the following more fundamental result concerning morphisms between compact nilspaces, proved in Section \ref{sec:genfactor}.
\begin{restatable}{thmrestate}{morfactor}\label{thm:morfactor}
\textit{Let $\ns,\nss$ be compact nilspaces, with $\nss$ of finite rank, and let $m:\ns\to \nss$ be a morphism. Then there exist a compact finite-rank nilspace $Q$, a fibration $\psi:\ns\to Q$, and a morphism $\psi':Q\to \nss$, such that $m=\psi'\co\psi$.}
\end{restatable}

\noindent Recall that a strict inverse system of compact nilspaces $\ns_i$, $i\in \mb{N}$, is a system of fibrations $(\psi_{i,j}:\ns_j\to\ns_i)_{i,j\in \mb{N}, i\leq j}$ such that $\psi_{i,i}=\id$ for all $i\in \mb{N}$ and $\psi_{i,j}\co \psi_{j,k} =\psi_{i,k}$ for all $i\le j\le k$ (see \cite[Definition 2.7.1]{Cand:Notes2}).

In Section \ref{sec:invlimthms}, we apply our results to give a swift proof of the following stronger version of the inverse limit theorem for compact nilspaces. \begin{restatable}{thmrestate}{transinvlim}\label{thm:transinvlim}
\textit{Let $\ns$ be a compact nilspace and let $H$ be a finite subset of $\tran(\ns)$. Then there is a strict inverse system $(\psi_{i,j}:\ns_j\to\ns_i)_{i,j\in \mb{N}, i\leq j}$ of compact finite-rank nilspaces $\ns_i$ such that $\ns=\varprojlim \ns_i$ and such that the limit maps $\psi_i:\ns\to\ns_i$ are all $\langle H\rangle$-consistent.}
\end{restatable}
\noindent As a special case we obtain that for every \emph{ergodic} nilspace system, if its corresponding group $H$ is a \emph{finitely-generated} subgroup of $\tran(\ns)$, then the nilspace system is an inverse limit of nilsystems; see Theorem \ref{thm:nilsysinvlim}. Thus we also provide a different proof of a result from \cite{GMV3}; see Remark \ref{rem:GMVrels}. Theorem \ref{thm:nilsysinvlim} is used in \cite{CScouplings} to extend the structure theorem of Host and Kra to finitely generated nilpotent group actions; see \cite[Theorem 5.12]{CScouplings}.

\section{Some motivating examples}\label{sec:exas}
\noindent We begin with a simple example showing that a fibration on a compact nilspace $\ns$ need not be $\alpha$-consistent for every $\alpha\in \tran(\ns)$.
\begin{example}\label{ex:conscex}
Recall the definition of the degree-$k$ nilspace structure $\mc{D}_k(\ab)$ on an abelian group $\ab$, in terms of the Gray-code alternating sum $\sigma_k$;  see \cite[Definition 2.2.30]{Cand:Notes1}. Let $\ns$ be the product nilspace $\mc{D}_1(\mb{Z}_2)\times \mc{D}_2(\mb{Z}_2)$ (where by $\mb{Z}_2$ we denote the 2-element group $\mb{Z}/2\mb{Z}$), and let $\nss$ be the nilspace $\mc{D}_2(\mb{Z}_2)$. Thus $\ns,\nss$ are 2-step compact finite-rank nilspaces (actually finite and with the discrete topology). Let $\psi$ denote the 2-nd coordinate projection $\ns\to\nss$, $(a,b)\mapsto b$. Using the third sentence in Definition \ref{def:fsmorph}, it is easily checked that $\psi$ is a fibration.

Let $\alpha:\ns\to\ns$ be the map $(a,b)\mapsto (a+1,b+a)$. We claim that $\alpha\in \tran(\ns)$. To see this, by \cite[Definition 3.2.27 and Lemma 3.2.13]{Cand:Notes1} it suffices to check that for every 3-cube $\q\in \cu^3(\ns)$, and any 2-face $F\subset \db{3}$, defining $\alpha^F(\q):\db{3}\to\ns$ by $\alpha^F(\q)(v)=\alpha(\q(v))$ for $v\in F$ and $\q(v)$ otherwise, we have $\alpha^F(\q)\in\cu^3(\ns)$. Let $\q':=\alpha^F(\q)$. By definition of the product nilspace structure, we have $\q'\in \cu^3(\ns)$ if and only if the coordinate projections $p_1:\ns\to \mc{D}_1(\mb{Z}_2)$ and $p_2:\ns\to \mc{D}_2(\mb{Z}_2)$ satisfy $p_i\co \q'\in \cu^3(\mc{D}_i(\mb{Z}_2))$ for $i=1,2$. By definition of $\mc{D}_1(\mb{Z}_2)$ and $\mc{D}_2(\mb{Z}_2)$, we therefore have $\q'\in \cu^3(\ns)$ if and only if the Gray-code alternating sum $\sigma_3(p_2\co \q')$ is 0 and for every 2-face map $\phi:\db{2}\to \db{3}$ we have $\sigma_2(p_1\co \q'\co \phi)=0$. Now from the definition of $\alpha$, we deduce that $\q'=\q+\q''$ where $\q''(v)=(1,p_1\co\q(v))$ if $v\in F$ and $\q''(v)=(0,0)$ otherwise. We then have $\sigma_3(p_2\co \q')=\sigma_3(p_2\co \q)+\sigma_3(p_2\co \q'')=0$, and also $\sigma_2(p_1\co \q'\co \phi)=\sigma_2(p_1\co \q\co \phi)+\sigma_2(p_1\co \q''\co \phi)=0$. This proves our claim. Note also that $\alpha$ is a minimal map.

Now observe that $\psi$ is not $\alpha$-consistent. Indeed, for example $(1,0),(0,0)\in \ns$ satisfy $\psi(1,0)=\psi(0,0)=0$, but $\psi\co \alpha(1,0)=1\neq 0=\psi \co \alpha(0,0)$.
\end{example}

\smallskip

\noindent One may try to avoid such examples by assuming that the fibration $\psi$ has additional properties. For instance, noting that if $\psi$ is injective then trivially it is $\alpha$-consistent, one may hope that if $\psi$ is ``sufficiently close" to being injective then it should also be $\alpha$-consistent. A way to capture closeness to injectivity could be to assume that every preimage $\psi^{-1}(y)$, $y\in \nss$ is a set of diameter\footnote{For a metric space $(\ns,d)$ and $B\subset \ns$, we define the \emph{diameter} of $B$ by $\diam(B):=\sup_{x,y\in B} d(x,y)$.} at most some small fixed $\delta>0$, for some fixed metric on $\ns$. However, one can elaborate on Example \ref{ex:conscex} to produce translations $\alpha$ such that even morphisms arbitrarily close to being injective in this sense can fail to be $\alpha$-consistent. Let us outline such a construction.

\begin{example}\label{ex:conscex2}
Let $\ns_0$ be the nilspace $\mc{D}_1(\mb{Z}_2)\times \mc{D}_2(\mb{Z}_2)$ from Example \ref{ex:conscex}. Let $\ns$ be the compact nilspace consisting of the power $\ns_0^{\mb{N}}$ with the product compact-nilspace structure. Let $\alpha_0\in\tran(\ns_0)$ be the translation $(a,b)\mapsto (a+1,b+a)$ from Example \ref{ex:conscex}. Let $\alpha$ denote the translation on $\ns$ defined by applying $\alpha_0$ to each coordinate of $x=(x_i)_{i\in \mb{N}}\in \ns$, i.e. $\alpha(x)=(\alpha_0(x_i))_{i\in \mb{N}}$. We can metrize $\ns$ with $d(x,y)=\sum_{i\in \mb{N}} 2^{-i}\,d_0(x_i,y_i)$ for $d_0$ the discrete metric on $\ns_0$. Consider now the following sequence of fibrations: for each $i\in \mb{N}$ let $\nss_i$ denote the product nilspace $\ns_0^i \times \mc{D}_2(\mb{Z}_2)$, and let $\psi_i: \ns \to \nss_i$, $x\mapsto (x_1,\ldots,x_i,\psi(x_{i+1}))$, where $\psi$ is the projection to the second coordinate on $\ns_0$ as in Example \ref{ex:conscex}. We then have the following facts (we leave the proofs to the reader):
\begin{enumerate}
\item For every $i\in\mb{N}$ the map $\psi_i$ is a (continuous) fibration.
\item We have $\sup_{y\in \nss_i} \diam(\psi_i^{-1}(y))=2^{-i}\to 0$ as $i\to \infty$. And yet,
\item For every $i$, letting $\mf{0}$ be the element of $\ns$ with all components equal to $(0,0)$, and $x$ the element with $x_j=(0,0)$ for $j\neq i+1$ and $x_{i+1}=(1,0)$, we have $\psi_i(x)=\psi_i(\mf{0})$ and $\psi_i\co\alpha(x)\neq \psi_i\co\alpha(\mf{0})$, so $\psi_i$ is not $\alpha$-consistent.
\end{enumerate}
\end{example}
\noindent Thus, Example \ref{ex:conscex2} shows that for $k>1$ there can be a translation on a $k$-step nilspace $\ns$ and fibrations $\psi:\ns\to\nss$ that are arbitrarily close to being injective  and yet still fail to be $\alpha$-consistent. As we explain in the sequel, if we are willing to \emph{refine} a fibration $\psi:\ns\to\nss$ by considering how $\psi$ factors through nilspaces finer than $\nss$, then the $\alpha$-consistency can be ensured for some such factor.

\section{Finite-rank valued morphisms factor through fibrations}\label{sec:genfactor}

\noindent In this section we prove Theorem \ref{thm:morfactor}, which we recall here for convenience. This is used in Section \ref{sec:transconsfact} to show that a fibration into a finite-rank nilspace always factors through some fibration consistent with a prescribed translation (Theorem \ref{thm:trancons}).
\morfactor*
\noindent Our proof of this theorem can be summarized simply as follows: we take the inverse limit expression $\ns=\varprojlim \ns_i$ given by \cite[Theorem 4]{CamSzeg} (see also \cite[Theorem 2.7.3]{Cand:Notes2}) and we show that, for some $i$ sufficiently large, the map $m$ factors through the limit map $\psi_i:\ns\to\ns_i$, so that we can set $\psi=\psi_i$. The proof uses some lemmas which we detail as follows.

Firstly, we have the following topological lemma.
\begin{lemma}\label{lem:topo}
Let $T, T_1, T_2,\ldots$ be compact metric spaces, and let $(\psi_i:T\to T_i)_{i\in \mb{N}}$ be a sequence of surjective continuous maps with the following properties:
\begin{enumerate}
\item For all $i\leq j$ and $x,y\in T$ with $\psi_j(x)=\psi_j(y)$, we have $\psi_i(x)=\psi_i(y)$.
\item For every $x\neq y$ there exists $i$ such that $\psi_i(x)\neq \psi_i(y)$.
\end{enumerate}
Let $(M,d)$ be a metric space and let $f:T\to M$ be continuous. Then for every $\epsilon>0$ there exists $i$ such that for every $x\in T_i$ we have $\diam\big(f(\psi_i^{-1}(x))\big)\leq \epsilon$.
\end{lemma}
\noindent The assumptions $(i)$, $(ii)$ in this lemma are satisfied in particular when $T$ is the topological inverse limit of the spaces $T_i$.
\begin{proof}
Suppose for a contradiction that for some $\epsilon>0$, for all $i\in \mb{N}$ there exist $x_i, y_i \in T$ such that $\psi_i(x_i)=\psi_i(y_i)$ and $d(f(x_i),f(y_i))\ge \epsilon$. Since $T$ is compact, we can assume (passing to subsequences if necessary) that there exist $x,y\in T$ with $x_i\to x$ and $y_i\to y$ as $i\to\infty$. By continuity of $f$ and $d$  we have $d(f(x),f(y))\ge \epsilon$, so in particular $x\neq y$.  By $(ii)$, this implies that $\psi_k(x)\neq \psi_k(y)$ for some $k$.  However, by assumption $\psi_j(x_j)=\psi_j(y_j)$ for every $j$, and if $j>k$ then by $(i)$ we therefore have $\psi_k(x_j)=\psi_k(y_j)$. Letting $j\to\infty$, by continuity of $\psi_k$ we deduce that $\psi_k(x)=\psi_k(y)$, a contradiction.
\end{proof}
Secondly, we have the following algebraic result, which is a basic fact about nilspaces.%
\begin{lemma}\label{lem:sub-nilspace}
Let $\ns,\nss$ be $k$-step nilspaces, and let $\psi:\ns\to\nss$ be a fibration. Then for every $y\in \nss$ the preimage $\psi^{-1}(y)$ is a sub-nilspace of $\ns$.
\end{lemma}
\begin{proof}
Of the three nilspace axioms (see \cite[Definition 1.2.1]{Cand:Notes1}), the composition and ergodicity axioms are clearly satisfied. The corner-completion axiom follows readily from the third sentence in Definition \ref{def:fsmorph} (using that for all $n$ the constant map $\db{n}\to \{y\}$ is a cube).\end{proof}
\noindent We now move on to the main element in the proof of Theorem \ref{thm:morfactor}, which is a lemma that extends the following result from \cite{CamSzeg} (see also \cite[Corollary 2.9.8]{Cand:Notes2}).
\begin{lemma}[\textrm{\cite[Corollary 3.2]{CamSzeg}}]\label{lem:restricted-rigidity}
For every compact finite-rank abelian group $\ab'$ and $j\in \mb{N}$, there exists $\epsilon>0$ such that the following holds. For every compact $k$-step nilspace $\ns$ and continuous morphism $m:\ns\to\mc{D}_j(\ab')$, if $\diam(m(\ns))\leq \epsilon$ then $m$ is constant.
\end{lemma}
The extension in question is the following.
\begin{lemma}\label{lem:rigidity-of-morphisms}
For every compact finite-rank $k$-step nilspace $\nss$, there exists $\delta>0$ such that the following holds. For every compact $k$-step nilspace $\ns$ and continuous morphism $m:\ns\to\nss$, if $\diam(m(\ns))\leq \delta$ then $m$ is constant.
\end{lemma}
\noindent For $i\le j$, we denote by $\pi_{i,j}:\nss_j\to \nss_i$ the projection between the two factors of $\nss$ (thus $\pi_i=\pi_{i,k}$ if $\nss$ is $k$-step). Our proof of Lemma \ref{lem:rigidity-of-morphisms} uses the following result.

\begin{proposition}\label{prop:rigidity}
Let $\nss$ be a compact finite-rank $k$-step nilspace, with a fixed metric $d_{\nss}$, and for each $i\in [k]$ let $\ab_i$ be the $i$-th structure group of $\nss$, with a fixed metric $d_{\ab_i}$. Then there is a collection of compact nilspace isomorphisms $\psi_{i,y}:\pi_{i-1,i}^{-1}(y) \to \mc{D}_i(\ab_i)$ for $i\in [k]$, $y\in \nss_{i-1}$, such that the following holds: for every $\epsilon>0$ there exists $\delta>0$ such that for all $a,b\in Y$ with $\pi_{i-1}(a)=\pi_{i-1}(b)=y$, if $d_{\nss}(a,b)< \delta$ then $d_{\ab_i}\big(\psi_{i,y}\co\pi_i(a),\psi_{i,y}\co\pi_i(b)\big)< \epsilon$.
\end{proposition}

\begin{proof}[Proof of Lemma \ref{lem:rigidity-of-morphisms} assuming Proposition \ref{prop:rigidity}] For each $i\in [k]$ let $\epsilon_i$ be the number $\epsilon(\ab_i,i)>0$ given by Lemma \ref{lem:restricted-rigidity}. Let $0<\epsilon < \min_{i\in [k]}\epsilon_i$, and apply Proposition \ref{prop:rigidity} to obtain the corresponding $\delta>0$ and functions $\psi_{i,y}$. We prove by induction on $i\in [0,k]$ that $\pi_i\co m $ is constant. The case $i=0$ is trivial since $\nss_0$ is the one-point space. We assume that $\pi_{i-1} \co m$ is constant, thus $\pi_{i-1}\co m(\ns) = y \in \nss_{i-1}$. By Proposition \ref{prop:rigidity} we have $\diam(\psi_{i,y}\co \pi_i\co m(\ns))<\epsilon<\epsilon(\ab_i,i)$. This together with the fact that $\psi_{i,y} \co \pi_i \co m$ is a morphism $\ns\to \mc{D}_i(\ab_i)$ implies, by Lemma \ref{lem:restricted-rigidity}, that this map is constant. Hence $\pi_i \co m$ is constant, since $\psi_{i,y}$ is injective.
\end{proof}

\begin{proof}[Proof of Proposition \ref{prop:rigidity}]
We first prove the case $i=k$. By \cite[Proposition A.1]{GMV2} there exists $\delta_k>0$ such that if $x\in\nss$ and $z\in \ab_k$ satisfy $d_{\nss}(x,x+z)<\delta_k$ then $d_{\ab_k}(0,z)<\epsilon$. Now for each fiber $\pi_{k-1}^{-1}(y)$, $y\in \nss_{k-1}$, fix any point $y'$ in this fiber and define $\psi_{k,y}:\pi_{k-1}^{-1}(y)\to \ab_k$ by $\psi_{k,y}(y'+z):=z$ (recalling from \cite[Theorem 3.2.19]{Cand:Notes1} that each point in this fiber is of the form $y'+z$ for a unique $z\in \ab_k$). Note that, since $d_{\ab_k}$ is translation-invariant, if $a,b$ are points in such a fiber $\pi_{k-1}^{-1}(y)$ with $d_{\nss}(a,b)<\delta_k$ and $a=y'+z_1$, $b=y'+z_2$, then $d_{\ab_k}(\psi_{k,y}(a),\psi_{k,y}(b))=d_{\ab_k}(z_1,z_2)=d_{\ab_k}(0, z_2-z_1)\le \epsilon$.

For $i\le k-1$ we argue similarly with $\nss_i$ instead of $\nss$, and with the same fixed $\epsilon>0$. Thus we obtain a function $\psi_{i,y}$ for each $y\in \nss_{i-1}$, and some $\delta'_i>0$ given by applying \cite[Proposition \ A.1]{GMV2}, with the property that for every $a,b\in \nss_i$ in the same fiber of $\pi_{i-1,i}$, if $d_{\nss_i}(a,b)\leq \delta_i'$ then $d_{\ab_i}\big(\psi_{i,y}(a),\psi_{i,y}(b)\big)\leq \epsilon$. Moreover, since $\pi_i:\nss \to \nss_i$ is a continuous function between compact metric spaces, it is uniformly continuous, so there exists $\delta_i>0$ such that if $d_{\nss}(a,b)<\delta_i$ then $d_{\nss_i}(\pi_i(a),\pi_i(b))<\delta'_i$.
 
Finally, we let $\delta=\min_{1\le i\le k} \delta_i$, and the result follows.
\end{proof}
We can now prove the main result of this section. \enlargethispage{0.1cm}
\begin{proof}[Proof of Theorem \ref{thm:morfactor}]
Let $\epsilon$ be the number $\delta$ given by Lemma \ref{lem:rigidity-of-morphisms} applied to $\nss$. Let $\ns=\varprojlim \ns_i$ be an inverse limit decomposition given by \cite[Theorem 4]{CamSzeg} (see also \cite[Theorem 2.7.3]{Cand:Notes2}); thus $\ns_i$ is a compact finite-rank nilspace and $\psi_i:\ns\to\ns_i$ is a fibration for each $i$. Applying Lemma \ref{lem:topo} with $T=\ns$, $T_i=\ns_i$ and $\epsilon$, we obtain $i_0$ such that $\diam(m(\psi_{i_0}^{-1}(x)))<\epsilon$ for all $x\in \ns_{i_0}$. We claim that we can set $Q:=\ns_{i_0}$ and $\psi:=\psi_{i_0}$. To prove this we show that there exists a morphism $\psi':Q\to \nss$ such that $m = \psi' \co \psi$. First note that setting $\psi'(x):= m(\psi^{-1}(x))$ gives a well-defined map $\psi':Q\to \nss$, because $m(\psi^{-1}(x))$ is a singleton for every $x\in Q$. Indeed $\psi^{-1}(x)$ is a sub-nilspace of $\ns$, by Lemma \ref{lem:sub-nilspace}, and $m$ restricted to this fiber is a morphism with image of diameter less than $\epsilon$, so by Lemma \ref{lem:rigidity-of-morphisms} this morphism is constant, so $\psi'$ is indeed well-defined. Moreover $\psi'$ is a morphism, since, by \cite[Lemma 3.3.9]{Cand:Notes1}, for every $\q\in \cu^n(Q)$ there exists $\q'\in \cu^n(\ns)$ such that $\q = \psi \co \q'$, so $\psi'\co \q = m \co \q'\in \cu^n(\nss)$. Finally $\psi'$ is continuous, for if $U$ is a closed subset of $\nss$ then, since $m^{-1}(U) = \psi^{-1} (\psi'^{-1}(U))$, and $\psi$ is surjective (see Remark \ref{rem:fibration}) and is closed \cite[p.\ 171, No.\ 6]{Munkres}, we have $\psi'^{-1}(U)=\psi(m^{-1}(U))$, a closed subset of $Q$.
\end{proof}

\section{Finite-rank-valued fibrations factor through translation-consistent fibrations}\label{sec:transconsfact}

\noindent In this section our main goal is to prove Theorem \ref{thm:trancons}, which we recall here.
\trancons*
\noindent Let us begin by proving Lemma \ref{lem:basic} from the introduction, which we recall here as well.
\basic*
\noindent Given a map $g:\ns\to\ns$, a map $\q:\db{n}\to \ns$, and a set $F\subset \db{n}$, we denote by $g^F(\q)$ the map $\db{n}\to \ns$ defined by $g^F(\q)(v)=g(\q(v))$ if $v\in F$ and $g^F(\q)(v)=\q(v)$ otherwise.
\begin{proof}
The $\alpha$-consistency implies clearly that $\beta$ is a well-defined map.

To see that $\beta$ is a translation, we check that \cite[Definition 3.2.27]{Cand:Notes1} holds: let $\q\in\cu^n(\nss)$ and $F$ be any face of codimension 1 in the cube $\db{n}$, and note that by fiber-surjectivity (see \cite[Lemma 3.3.9]{Cand:Notes1}) there is $\q'\in\cu^n(\ns)$ such that $\psi\co \q'=\q$. Hence $\beta^F(\q)=\psi\co\big(\alpha^F (\q')\big)$. This equality implies that $\beta^F(\q)\in \cu^n(\nss)$, since $\psi$ is a morphism and $\alpha^F(\q')\in \cu^n(\ns)$. 

Moreover, the translation $\beta$ is \emph{continuous}, for if $C\subset \nss$ is closed then $\beta^{-1}(C)=\psi((\psi \co \alpha)^{-1}(C))$ is closed, since $\psi \co \alpha$ is continuous and $\psi$ is a closed surjective map by \cite[p. 171]{Munkres} and Remark \ref{rem:fibration}. 

The last sentence of the lemma is straightforward to check.
\end{proof}

\noindent We now turn to the proof of Theorem \ref{thm:trancons}. Our strategy is to obtain this as a consequence of Theorem \ref{thm:morfactor}. We use the following notation. Given maps $\psi_i:\ns\to\nss_i$, $i\in [n]$, we denote by $\Delta(\psi_1,\ldots,\psi_n)$ their \emph{diagonal product} $\ns\to\nss_1\times \cdots\times \nss_n$, defined by $\Delta(\psi_1,\ldots,\psi_n)(x)= (\psi_1(x),\ldots,\psi_n(x))$, and we denote by $\psi_1\times \cdots\times \psi_n$ their \emph{product} $\ns^n\to\nss_1\times \cdots\times \nss_n$, $(x_1,\ldots,x_n)\mapsto (\psi_1(x_1),\ldots,\psi_n(x_n))$. 

A first fact that follows from Theorem \ref{thm:morfactor}, and which we use for further results in this section, is the following ``common-refinement" lemma.
\begin{lemma}\label{lem:merging}
Let $\ns$, $Q_1, Q_2, \ldots, Q_d$ be compact nilspaces, with $Q_i$ of finite rank for each $i\in [d]$, and let $m_i:\ns\to Q_i$ be a fibration for each $i$. Then there is a compact finite-rank nilspace $Q$ and fibrations $m:\ns\to Q$, $m_i':Q\to Q_i$ such that $m_i=m_i'\co m$ for each $i\in [d]$.
\end{lemma}\newpage
\noindent The lemma yields the following commutative diagram:
\vspace{-0.7cm}
\begin{flushright} 
    \includegraphics[width=0.27\textwidth]{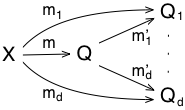} \qquad\mbox{}
\end{flushright}
\begin{proof}
Let $m'=\Delta(m_1,\ldots,m_d)$, and note that this map is a continuous morphism from $\ns$ to the product nilspace $Q_1\times \cdots\times Q_d$. Applying Theorem \ref{thm:morfactor} to $m'$, we obtain a compact finite-rank nilspace $Q$, a fibration $m:\ns\to Q$, and a morphism $\phi:Q\to Q_1\times\cdots\times Q_d$, such that $m'= \phi\co m$. We then set $m_i'=p_i \co \phi$ for $i\in [d]$, where $p_i:Q_1\times\cdots\times Q_d\to Q_i$ are the coordinate projections, which are fibrations. We claim that, since $m_i=m_i'\co m$ and both $m_i$ and $m$ are fiber-surjective, each $m_i'$ is also fiber-surjective. To see this, fix any fiber $F'=\pi_{n,Q}^{-1}(x')$, and note that since $m$ is fiber-surjective there exists a fiber $F=\pi_{n,\ns}^{-1}(x) $ such that $m(F)=F'$. Thus $m_i'(F')=m_i(F)$, and since $m_i$ is fiber-surjective, we have that $m_i'(F')$ is a $\pi_{n,Q_i}$-fiber, as claimed.
\end{proof}
\noindent For two maps $\psi_i:\ns\to\nss_i$, $i=1,2$ defined on all of $\ns$ (but with $\nss_1$, $\nss_2$ possibly different spaces), we write $\psi_1\lesssim \psi_2$ if the partition generated by the latter map refines the partition generated by the former, i.e. if the partitions $\mc{P}_i:=\{\psi_i^{-1}(y):y\in\psi_i(\ns)\}$, $i=1,2$ satisfy that every set in $\mc{P}_1$ is a union of some sets in $\mc{P}_2$. We write $\psi_1\approx \psi_2$ to mean that $\psi_1\lesssim \psi_2$ and $\psi_2\lesssim \psi_1$ both hold. If $\ns,\nss$ are $k$-step nilspaces and $\psi:\ns\to\nss$ is a morphism, then for each $i\in[k]$ there is a morphism $\pi_i(\ns)\to\pi_i(\nss)$, which we denote by $(\psi)_{(i)}$, such that $(\psi)_{(i)}\co \pi_{i\ns}= \pi_{i,\nss}\co \psi$; see \cite[Definition 3.3.1 and Proposition 3.3.2]{Cand:Notes1}. It is seen straight from the definitions that if $\psi$ is a fibration then so is $(\psi)_{(i)}$ for each $i$.

Our proof of Theorem \ref{thm:trancons} works by induction on the step $k$. A key ingredient in the induction is Lemma \ref{lem:congrefine} below. That lemma in turn relies on the fiber-product construction in the category of compact nilspaces, which we detail as follows.
\begin{lemma}\label{lem:fp}
Let $\ns^{(1)},\ns^{(2)},\ns^{(3)}$ be compact nilspaces and let $\psi_1:\ns^{(1)}\to \ns^{(3)}$, $\psi_2:\ns^{(2)}\to \ns^{(3)}$ be fibrations. Then the fiber-product $\ns^{(1)}\times_{\ns^{(3)}}\ns^{(2)}:=\{(x_1,x_2)\in \ns^{(1)}\times \ns^{(2)}: \psi_1(x_1)=\psi_2(x_2)\}$ is a compact sub-nilspace of the product nilspace\footnote{The definition of a product nilspace may be recalled from \cite[Definition 3.1.2]{Cand:Notes1}.} $\ns^{(1)}\times \ns^{(2)}$.
\end{lemma}
\begin{proof}
Let $Q=\ns^{(1)}\times_{\ns^{(3)}}\ns^{(2)}$. We have to check that, if for each $n\geq 0$ we let $\cu^n(Q)$ consist of the $Q$-valued cubes in $\cu^n(\ns^{(1)}\times \ns^{(2)})$, then these cube sets $\cu^n(Q)$ satisfy the nilspace axioms. The composition and ergodicity axioms are easily verified. Let us check the corner-completion axiom. Let $\Delta(\q_1', \q_2')\in \cor^n(Q)$,  which implies that $\q_1'\in \cor^n(\ns^{(1)})$ and $\q_2'\in \cor^n(\ns^{(2)})$. Let $\q_2\in \cu^n(\ns^{(2)})$ be a completion of $\q_2'$. By definition of $Q$ we have $\psi_1 \co \q_1'(v)=\psi_2 \co \q_2(v)$ for all $v\not=1^n$. Therefore $\psi_2 \co \q_2$ is a completion of the corner $\psi_1 \co \q_1'$, so, since $\psi_1$ is a fibration, there exists $\q_1\in \cu^n(\ns_1)$ that completes $\q_1'$ such that $\psi_1 \co \q_1=\psi_2 \co \q_2$. Therefore $\Delta(\q_1, \q_2)$ completes $\Delta(\q_1',\q_2')$.
\end{proof}
\begin{lemma}\label{lem:congrefine}
Let $\ns$ be a $k$-step compact nilspace, let $\psi_1:\ns\to\nss$ be a fibration, let $\psi_2:\ns\to W$ be a fibration that factors through $\pi_{k-1,\ns}$, let $\psi_3:W\to\nss_{k-1}$ be a fibration such that $\pi_{k-1,\nss}\co \psi_1 = \psi_3\co\psi_2$, and let $\psi:=\Delta(\psi_1,\psi_2)$. Then $\psi$ is a fibration $\ns\to \nss\times_{\nss_{k-1}} W$. Moreover $\big(\psi)_{(k-1)}\approx (\psi_2)_{(k-1)}$, and if $W$ and $\nss$ are of finite rank then so is $\psi(\ns)$.
\end{lemma}

\vspace{0.8cm}

\noindent The following diagram illustrates the assumptions: \vspace{-1.5cm}
\begin{flushright} 
\includegraphics[width=0.33\textwidth]{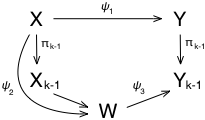}\end{flushright}
\begin{proof}
Let $Q$ denote the fiber-product nilspace $\nss\times_{\nss_{k-1}} W$ for the fibrations $\pi_{k-1,\nss},\psi_3$. We claim that $\psi(\ns)=Q$. The inclusion $\psi(\ns) \subseteq Q$ is clear since $\pi_{k-1,\nss}\co \psi_1 = \psi_3\co\psi_2$. For the opposite inclusion, let $(a,b)\in Q$  and $x\in \ns$ be any element with $\psi_2(x)=b$ (such an $x$ exists by the surjectivity of the fibration $\psi_2$). Letting $\ab_k(\ns)$ denote the $k$-th structure group of $\ns$, for every $z\in \ab_k(\ns)$ we have $\psi_2(x+z)=\psi_2(x)$, since $\psi_2$ factors through $\pi_{k-1,\ns}$. Now $\pi_{k-1}(a)=\psi_3(b)=\psi_3(\psi_2(x))=\pi_{k-1}(\psi_1(x))$. Thus there exists $z'\in \ab_k(\nss)$ such that $a=\psi_1(x)+z'=\psi_1(x+z)$ for some $z\in \ab_k(\ns)$, where the last equality follows from the fiber-surjectivity of $\psi_1$. Hence $\psi(x+z)=(a,b)$, and the  inclusion follows.

Note that from the definitions it is clear that $\psi$ is a morphism $\ns\to Q$. We now prove that $\psi$ is a fibration. Let $\q'\in\cor^n(\ns)$, and let $\tilde \q\in \cu^n(Q)$ be a completion of $\psi\co \q'$, thus $\tilde\q=\Delta(\q_1,\q_2)$ for $\q_1\in \cu^n(\nss)$, $\q_2\in \cu^n(W)$ completing $\psi_1\co \q'$, $\psi_2\co \q'$ respectively, and satisfying $\psi_3\co\q_2=\pi_{k-1,\nss}\co \q_1$. Since $\psi_2$ is a fibration, we can complete $\q'$ to $\q_3 \in \cu^n(\ns)$ with $\psi_2\co \q_3 = \q_2$. Now $\pi_{k-1,\nss}\co \psi_1\co \q_3=\psi_3\co \psi_2\co \q_3=\psi_3\co\q_2=\pi_{k-1,\nss}\co \q_1$, so $\psi_1\co \q_3$ and $\q_1$ are in the same $\pi_{k-1,\nss}$-fiber, so there is $\q''\in \cu^n(\mc{D}_k(\ab_k(\nss)))$ such that $\psi_1\co \q_3 + \q'' = \q_1$. Since $\psi_1$ is a fibration, its $k$-th structure morphism\footnote{The notion of the \emph{structure morphisms} of a nilspace morphism may be recalled from \cite[Definition 3.3.1]{Cand:Notes1}.} $\ab_k(\ns)\to \ab_k(\nss)$ is surjective, whence there is $\q_4\in \cu^n(\mc{D}_k(\ab_k(\ns)))$ such that $\psi_1\co (\q_3+\q_4)=\psi_1\co \q_3 +\q'' =\q_1$, and since $\psi_2$ factors through $\pi_{k-1,\ns}$ we still have  $\psi_2\co (\q_3+\q_4)=\q_2$, so $\q_3+\q_4$ completes $\q'$ as required.

To prove the remaining claims in the lemma, it is useful first to give a more precise description of $Q$ in terms of $\nss$ and $W$. We first show that $Q_{k-1}$ is isomorphic to $W$ as a compact nilspace. The isomorphism is given by the map $\varphi:W\to Q_{k-1}$ defined by $\varphi(b)=\pi_{k-1,Q}(a,b)$, for any $a\in \nss$ such that $\pi_{k-1}(a)=\psi_3(b)$. This is well-defined because if $a,a'\in \nss$ satisfy $\pi_{k-1}(a)=\pi_{k-1}(a')$ then $a'=a+z$ for some $z\in \ab_k(\nss)$, and then from basic properties of the relation $\sim_{k-1}$ (see \cite[Lemma 3.2.4, Remark 3.2.12]{Cand:Notes1}) it follows that $\pi_{k-1,Q}(a,b)=\pi_{k-1,Q}(a+z,b)$. The map $\varphi$ is injective,  because if $\varphi(b)=\varphi(b')$ then the definition of $\sim_{k-1}$ on $Q$ and the fact that $W$ is $(k-1)$-step imply that $b=b'$. By definition of $\varphi$ and the fact that $\pi_{k-1,Q}$ is surjective, we also have that $\varphi$ is surjective. Let us now check that $\varphi$ and $\varphi^{-1}$ are both  morphisms. To see that $\varphi$ is a morphism, note that if $\q\in\cu^n(W)$ then $\psi_3 \co \q\in \cu^n(\nss_{k-1})$ and there exists $\q'\in \cu^n(\nss)$ such that $\pi_{k-1}\co \q'=\psi_3 \co \q$, whence $\varphi \co \q = \pi_{k-1,Q} \co\big(\Delta(\q', \q)\big) \in \cu^n(Q_{k-1})$. That $\varphi^{-1}$ is a morphism follows from the definition of $\varphi$ and $Q$ and the fact that $\pi_{k-1,Q}$ is a fibration. Let us show that $\varphi^{-1}$ is continuous.  Let $p:Q\to W$ be the projection $(a,b)\mapsto b$, which is continuous and satisfies $\varphi^{-1} \co \pi_{k-1,Q}=p$. Then for any open set $U\subset Q_{k-1}$, we have $\pi_{k-1,Q}^{-1}(\varphi(U))=p^{-1}(U)$, and since $\pi_{k-1,Q}$ is surjective we have $\varphi(U)=\pi_{k-1,Q}(p^{-1}(U))$. Since $p$ is continuous and $\pi_{k-1,Q}$ is open (see \cite[Remark 2.1.7]{Cand:Notes2}), we have that $\varphi(U)$ is open, and the continuity of $\varphi^{-1}$ follows. We thus have a continuous bijection $\varphi^{-1}$ between compact metric spaces, so $\varphi$ is a homeomorphism. Having 
 shown that $Q_{k-1}$ is isomorphic to $W$, let us now complete the  description of $Q$ by describing $\ab_k(Q)$. We claim that $\ab_k(Q)$ is isomorphic as a compact abelian group to $\ab_k(\nss)$. To see this, it suffices to show that for any fiber $F$ of $\sim_{k-1}$ on $Q$, as a compact sub-nilspace of $Q$ this fiber is isomorphic to $\mc{D}_k(\ab_k(\nss))$ (the isomorphism of the structure groups then follows from known theory; see for instance the end of the proof of \cite[Proposition 2.1.9]{Cand:Notes2}). Fix $(a_0,b_0)\in F$ and note that by definition of $\sim_{k-1}$ on $Q$ we have that every $(a,b)\in F$ is $(a_0+z,b_0)$ for some unique $z\in \ab_k(\nss)$. Let $\tau:F\to \ab_k(\nss)$ be the map sending $(a,b)$ to this unique $z$. Using that $Q$ is a sub-nilspace of $\nss\times W$, it is checked in a straightforward way that $\tau$ is a compact nilspace isomorphism $F\to \mc{D}_k\big(\ab_k(\nss)\big)$ ($\tau$ is clearly a bijection, and the cube-preserving properties can be checked using \cite[(2.9)]{Cand:Notes1}).

We can now prove the last claims in the lemma. From the definition of the isomorphism $\varphi$ above and the assumption $\pi_{k-1,\nss}\co\psi_1=\psi_3\co\psi_2$, it follows that $(\psi)_{(k-1)}\co\pi_{k-1,\ns}=\varphi \co \psi_2$, and from this it is easily deduced that $\big(\psi)_{(k-1)}\approx (\psi_2)_{(k-1)}$. Finally, by the above description of $Q$ it is clear that if $W$ and $\nss$ are of finite rank then all the structure groups of $Q$ are Lie groups, so $Q$ is also of finite rank. 
\end{proof} 

\noindent For a map $f:A\to B$ we write $a\sim_f a'$ if $f(a)=f(a')$. Our proof of Theorem \ref{thm:trancons} uses the next fact.
\begin{lemma}\label{lem:small}
Let $\psi:\ns\to\nss$ and $R:\ns\to\nss'$ be fibrations with $\psi\lesssim R$, let $\phi_k$ be the $k$-th structure morphism of $\psi$, and let $\varphi= \Delta\big(\psi, (R)_{(k-1)}\co\pi_{k-1,\ns}\big)$. If $x\sim_\varphi y$ then $x\sim_R y+z$ for some $z\in \ker(\phi_k)$.
\end{lemma}
\begin{proof}
If $x\sim_\varphi y$ then $\psi(x)=\psi(y)$ and $R(x) \sim_{k-1} R(y)$ (since $(R)_{(k-1)}\co\pi_{k-1,\ns}=\pi_{k-1,\nss'}\co R$). Then, since $R$ is a fibration, there exists $z\in \ab_k(\ns)$ such that $R(x)=R(y+z)$. Since $\psi \lesssim R$, we deduce that $\psi(x)=\psi(y+z)=\psi(x)+\phi_k(z)$. Hence $z\in \ker(\phi_k)$.
\end{proof}

\begin{proof}[Proof of Theorem \ref{thm:trancons}]
We argue by induction on $k$. The result is trivial for $k=0$. Let $k>0$ and assume that the result holds for $k-1$.

By Lemma \ref{lem:merging}, there is a finite-rank nilspace $Q'$, fibrations $q':\ns\to Q'$, $m':Q'\to \nss'$, and $m'_\alpha:Q'\to \nss'$ for each $\alpha\in H$, with $\psi'=m'\co q'$ and $\psi'\co\alpha=m'_\alpha\co q'$ for each $\alpha$. Note that $x\sim_{q'} x'$ implies that $x\sim_{\psi'\co\alpha} x'$ for each $\alpha$. Let $H'=\{(\alpha)_{(k-1)}:\alpha\in H\}$. Let $q_2:\ns_{k-1}\to W$ be an $H'$-consistent fibration on $\ns_{k-1}$ obtained by applying the inductive hypothesis to $(q')_{(k-1)}:\ns_{k-1}\to Q'_{k-1}$ (in particular $q_2\gtrsim (q')_{(k-1)}$), and let $p:W\to Q'_{k-1}$ be the fibration such that $p\co q_2 =(q')_{(k-1)}$. Let $\psi_3$ denote the map $(m')_{(k-1)}\co p: W\to \nss'_{k-1}$ (thus, in this inductive application of Theorem \ref{thm:trancons}, the objects $\psi,\nss$ from the conclusion of the theorem are denoted here by $q_2,W$ respectively). Let $\psi_2:=q_2\co \pi_{k-1,\ns}$. Note that for each $\alpha$, since $q_2$ is $(\alpha)_{(k-1)}$-consistent and $\pi_{k-1,\ns}$ is $\alpha$-consistent, we have that $\psi_2$ is $\alpha$-consistent. The following diagram illustrates the situation:
\begin{center}
\includegraphics[width=0.35\textwidth]{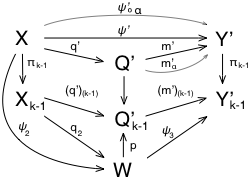}
\end{center}
We now claim that $\psi:=\Delta(\psi',\psi_2)$ satisfies the required conclusion. We know that $\psi$ is a fibration, by Lemma \ref{lem:congrefine} applied with $\psi_1=\psi'$ and $\psi_2,\psi_3$ defined above. We also know by that lemma that, since $W$ and $\nss'$ are of finite rank, so is $\psi(\ns)$.

Let us check that $\psi$ is $H$-consistent. Fix any $\alpha\in H$ and suppose that $x\sim_{\psi} y$. We have to show that $\alpha(x)\sim_{\psi} \alpha(y)$. Firstly we claim that $\alpha(x)\sim_{\psi'} \alpha(y)$. To see this, note that since $(q')_{(k-1)}\lesssim q_2$, we have $\Delta(\psi', (q')_{(k-1)}\co\pi_{k-1,\ns}) \lesssim \Delta(\psi', q_2\co \pi_{k-1,\ns})=\psi$, and then by Lemma \ref{lem:small} applied to $\psi'$ and $R=q'$, it follows that $q'(x)= q'(y+z)$ for some $z\in\ker(\phi_k)$ (where $\phi_k$ is the $k$-th structure morphism of $\psi'$). Applying $m'_{\alpha}$ to both sides of the last equality, we deduce that $\psi'(\alpha(x))=\psi'(\alpha(y+z))$. We now use basic properties of $k$-step nilspaces, namely that translations commute with the action of $\ab_k$ \cite[Lemma 3.2.37]{Cand:Notes1}, and the equivariance property involving $\phi_k$ given by \cite[Proposition 3.3.2 and Definition 3.3.1 (ii)]{Cand:Notes1}, to deduce that $\psi'(\alpha(y+z))=\psi'(\alpha(y)+z)=\psi'(\alpha(y))+\phi_k(z)=\psi'(\alpha(y))$, so $\alpha(x)\sim_{\psi'} \alpha(y)$ as claimed. Then, the $\alpha$-consistency of $\psi_2$ (seen above) implies $\alpha(x)\sim_{\Delta(\psi', \psi_2)} \alpha(y)$, so $\psi$ is $\alpha$-consistent. Since this holds for all $\alpha\in H$, the result follows.
\end{proof}

\section{An inverse limit theorem for nilspace systems}\label{sec:invlimthms}

\noindent As recalled in the introduction, one of the main results in  nilspace theory is the inverse limit theorem, which states that every compact nilspace is an inverse limit of compact finite-rank nilspaces. In this section we prove the following stronger version of this result.
\vspace{-0.7cm}\transinvlim*
\begin{proof}
Recall that by the inverse limit theorem \cite[Theorem 2.7.3]{Cand:Notes2} there is a strict inverse system $\{\psi_{i,j}':\ns_j'\to\ns_i'\}$ such that $\ns=\varprojlim \ns_i'$. We use the following inductive argument to  upgrade this inverse system gradually. Starting with the limit map $\psi_1':\ns \to \ns_1'$, we use Theorem \ref{thm:trancons} to obtain a finite-rank nilspace $\ns_1$, an $H$-consistent fibration $\psi_1:\ns \to \ns_1$ and a fibration $q_1:\ns_1\to \ns_1'$ with $\psi_1'=q_1\co\psi_1$. Suppose now that we have upgraded the system up to $i$, thus we have $H$-consistent fibrations $\psi_j:\ns\to\ns_j$, fibrations $q_j:\ns_j\to\ns_j'$, $j\in [i]$, and also fibrations $\psi_{j,k}:\ns_k\to\ns_j$ for $1\leq j\leq k\leq i$. Then we apply Lemma \ref{lem:merging} to $\psi_i$ and $\psi_{i+1}'$ to obtain a fibration $\psi_{i+1}'':\ns\to\ns_{i+1}''$ through which $\psi_i$ and $\psi_{i+1}'$ both factor. Then we apply Theorem \ref{thm:trancons} to $\psi_{i+1}''$ to refine it to an $H$-consistent fibration $\psi_{i+1}:\ns\to\ns_{i+1}$. Continuing this way, the result follows.
\end{proof}
\noindent The following diagram illustrates the inductive step: \vspace{-0.7cm}
\begin{flushright}
    \includegraphics[width=0.4\textwidth]{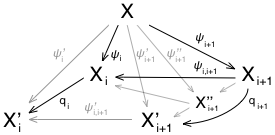}
\end{flushright}
\noindent Recall that a measure-preserving action of a countable discrete group $G$ on a probability space $(\Omega,\mc{A},\mu)$ is \emph{ergodic} if $\forall\,A\in\mc{A}$, $(\forall\,g\in G,\, \mu( (g\cdot A) \Delta A)=0)\Rightarrow  \mu(A)\in\{0,1\}$.

We obtain the following consequence of Theorem \ref{thm:transinvlim}.

\begin{theorem}\label{thm:nilsysinvlim}
Let $\ns$ be a $k$-step compact nilspace and let $H$ be a finitely generated subgroup of $\tran(\ns)$ acting ergodically on $\ns$ \textup{(}relative to the Haar probability measure on $\ns$\textup{)}. Then the nilspace system $(\ns,\mu,H)$ is an inverse limit of $k$-step nilsystems.
\end{theorem}
\noindent This follows immediately by combining Theorem \ref{thm:transinvlim} with the following basic lemma.
\begin{lemma}\label{lem:transact}
Let $\nss$ be a $k$-step compact finite-rank nilspace, and let $H$ be a finitely generated subgroup of $\tran(\nss)$ acting ergodically on $\nss$. Then the $k$-step nilpotent Lie group $\langle\tran(\nss)^0, H\rangle$ acts transitively on $\nss$.
\end{lemma}
\begin{remark}\label{rem:conn}
Note that no assumption on $\nss$ is made other than that it is of finite rank, hence no additional assumption on $\ns$ is needed in Theorem \ref{thm:nilsysinvlim}. In Lemma \ref{lem:transact} the ergodicity of the action of $H$ indeed suffices to deduce the claimed transitivity, but to show this we use a deep result, namely the transitivity of the action of the identity component $\tran(\nss)^0$ on each connected component of $\nss$, established in \cite[Corollary 3.3]{CamSzeg}.
\end{remark}
\begin{proof}
By \cite[Corollary 3.3]{CamSzeg} (see also \cite[Corollary 2.9.12]{Cand:Notes2}), if there is only one component in $\nss$ then we are done. Suppose, then, that there are at least two components. It suffices to prove the claim that for every two components $C,C'\subset \nss$ there is $g\in H$ such that $\mu(g(C)\cap C')>0$. Indeed, if this holds then for any $y,y'\in \nss$ there is $g'\in H$, and some $x$ in the component $C(y)$ containing $y$, such that $g'\cdot x\in C(y')$. Then, by the transitivity of $\tran(\nss)^0$ on $C(y),C(y')$, there are $\beta,\beta'\in \tran(\nss)^0$ with $\beta(y)=x$ and $\beta'\cdot g'\cdot x=y'$, so the element $g= \beta' g' \beta$ satisfies $g\cdot y=y'$, and the transitivity of $\langle\tran(\nss)^0, H\rangle$ follows.

To prove the claim, note first that since the compact nilspace $\nss$ has finite rank, it is a finite-dimensional manifold (see \cite[Lemma 2.5.3]{Cand:Notes2}) and is therefore locally connected, so each of its connected components is an open set \cite[Theorem 25.3]{Munkres}. Since the Haar measure $\mu$ on $\nss$ is strictly positive \cite[Proposition 2.2.11]{Cand:Notes2}, every component has positive measure. Let $C,C'$ be any two such components, and suppose for a contradiction that for every $g\in H$ we have $\mu(g(C)\cap C')=0$. Then the $H$-invariant set $\cup_{g\in H} g\cdot C$ is disjoint from $C'$ up to a $\mu$-null set, so it is an $H$-invariant set of measure strictly between 0 and 1, contradicting the ergodicity of $H$.
\end{proof}

\begin{remark}\label{rem:GMVrels}
As mentioned in the introduction, the results in this section are related to a result of Gutman, Manners and Varj\'u, namely the version of \cite[Theorem 1.29]{GMV3} for nilspace systems  mentioned in their paper after their Theorem 1.29. We obtain Theorem \ref{thm:nilsysinvlim} as a swift consequence of the fundamental factorization results for morphisms from previous sections, and this makes our proof markedly different from the arguments in \cite{GMV3}. Moreover, Theorem \ref{thm:nilsysinvlim} is a special case of Theorem \ref{thm:transinvlim}, which is a direct generalization of the inverse limit theorem \cite[Theorem 4]{CamSzeg}, and which concerns arbitrary (not just ergodic) finitely generated nilspace systems, whereas \cite[Theorem 1.29]{GMV3} concerns minimal group actions. Thus the results in this section and \cite[Theorem 1.29]{GMV3} are complementary.
\end{remark}

\noindent \textbf{Acknowledgements.} We are very grateful to the anonymous referee for useful remarks that helped to improve this paper.


\begin{thebibliography}{1}
\bibitem{AGH} L. Auslander, L. Green, F. Hahn, \emph{Flows on homogeneous spaces}, Annals of Mathematics Studies, No. 53 Princeton University Press, Princeton, N.J. 1963.
\bibitem{CamSzeg} O. A. Camarena, B. Szegedy, \emph{Nilspaces, nilmanifolds and their morphisms}, preprint. \href{http://arxiv.org/abs/1009.3825}{arXiv:1009.3825}
\bibitem{Cand:Notes1} P. Candela, \emph{Notes on nilspaces: algebraic aspects}, Discrete Analysis, 2017, Paper No. 15, 59 pp.
\bibitem{Cand:Notes2} P. Candela, \emph{Notes on compact nilspaces}, Discrete Analysis, 2017, Paper No. 16, 57pp.
\bibitem{CScouplings} P. Candela, B. Szegedy, \emph{Nilspace factors for general uniformity seminorms, cubic exchangeability and limits}, preprint. \href{http://arxiv.org/abs/1803.08758}{arXiv:1803.08758}
\bibitem{Furstentorus} H. Furstenberg, \emph{Strict ergodicity and transformation of the torus}, Amer. J. Math. \textbf{83} (1961), 573--601.
\bibitem{GGY} E. Glasner, Y. Gutman, X. Ye, \emph{Higher order regionally proximal equivalence relations for general minimal group actions}, Adv. Math. \textbf{333} (2018), 1004--1041.
\bibitem{GMV1} Y. Gutman, F. Manners, P. P. Varj\'u, \emph{The structure theory of nilspaces I}, to appear in J. Anal. Math. 
\bibitem{GMV2} Y. Gutman, F. Manners, P. P. Varj\'u, \emph{The structure theory of nilspaces II: Representation as nilmanifolds}, Trans. Amer. Math. Soc., article electronically published on October 1, 2018.
\bibitem{GMV3} Y. Gutman, F. Manners, P. P. Varj\'u, \emph{The structure theory of nilspaces III: Inverse limit representations and topological dynamics}, preprint. \href{http://arxiv.org/abs/1605.08950}{arXiv:1605.08950}
\bibitem{HKbook} B. Host, B. Kra, \emph{Nilpotent structures in ergodic theory}, Mathematical Surveys and Monographs, Volume 236, American Mathematical Society, Providence, Rhode Island, 2018.
\bibitem{HK} B. Host, B. Kra, \emph{Nonconventional ergodic averages and nilmanifolds}, Ann. of Math. (2) \textbf{161} (2005), no. 1, 397--488.
\bibitem{HKM} B. Host, B. Kra, A. Maass, \emph{Nilsequences and a structure theorem for topological dynamical systems}, Adv. Math. \textbf{224} (2010), no. 1, 103--129.
\bibitem{Malcev} A. I. Mal'cev, \emph{On a class of homogeneous spaces}, Izv. Akad. Nauk. SSSR. Ser. Mat. \textbf{13} (1949), 9--32.
\bibitem{Munkres} J. R. Munkres, \emph{Topology, Second Edition}, Prentice Hall, Inc., Upper Saddle River, NJ, 2000.
\bibitem{Parry} W. Parry, \emph{Ergodic properties of affine transformations and flows on nilmanifolds}, Amer. J. Math. \textbf{91} (1969), 757--771.
\bibitem{Ziegler} T. Ziegler, \emph{Universal characteristic factors and Furstenberg averages}, J. Amer. Math. Soc. \textbf{20} (2007), no. 1, 53--97.
\end{thebibliography}
\end{document}